\documentclass[reqno, 12pt]{amsart}

\pdfoutput=1
\makeatletter
\let\origsection=\section \def\section{\@ifstar{\origsection*}{\mysection}} 
\def\mysection{\@startsection{section}{1}\z@{.7\linespacing\@plus\linespacing}{.5\linespacing}{\normalfont\scshape\centering\S}}
\makeatother    

\usepackage{amsmath,amssymb,amsthm}
\usepackage[foot]{amsaddr}
\usepackage{mathrsfs}
\usepackage{mathtools}
\usepackage{mathabx}\changenotsign
\usepackage{dsfont}
\usepackage{nicefrac}
\usepackage[normalem]{ulem}
\usepackage{graphicx}
\usepackage{xcolor}
\usepackage{tikz}
\usepackage[foot]{amsaddr}
\usepackage{verbatim}

\usepackage{enumitem}

\usepackage{lineno}
\usepackage{hyperref}
\usepackage{cleveref}
\hypersetup{
    colorlinks,
    linkcolor={red!60!black},
    citecolor={green!60!black},
    urlcolor={blue!60!black},
}

\usepackage[open,openlevel=2,atend]{bookmark}

\input{_preamble/myBiblatex}
\defbibheading{bibliography}{%
  \section*{References}%
}
\usepackage[T1]{fontenc}
\usepackage{lmodern}
\usepackage[babel]{microtype}
\usepackage[english]{babel}

\usepackage{geometry}
\numberwithin{equation}{section}
\numberwithin{figure}{section}

\let\polishlcross=\l
\def\l{\ifmmode\ell\else\polishlcross\fi}

\def\paragraph#1{%
  \smallskip
  \noindent\textbf{#1.}\enspace}

\makeatletter
\def\moverlay{\mathpalette\mov@rlay}
\def\mov@rlay#1#2{\leavevmode\vtop{   \baselineskip\z@skip \lineskiplimit-\maxdimen
   \ialign{\hfil$\m@th#1##$\hfil\cr#2\crcr}}}
\newcommand{\charfusion}[3][\mathord]{
    #1{\ifx#1\mathop\vphantom{#2}\fi
        \mathpalette\mov@rlay{#2\cr#3}
      }
    \ifx#1\mathop\expandafter\displaylimits\fi}
\makeatother

\usepackage{thmtools, thm-restate}

\theoremstyle{plain}
\newtheorem{thm}{Theorem}[section]
    \crefname{thm}{Theorem}{Theorems}
\newtheorem{theorem}[thm]{Theorem}
    \crefname{theorem}{Theorem}{Theorems}
\newtheorem{lemma}[thm]{Lemma}
    \crefname{lemma}{Lemma}{Lemmas}
\newtheorem{lem}[thm]{Lemma}
    \crefname{lem}{Lemma}{Lemmas}
\newtheorem{corollary}[thm]{Corollary}
    \crefname{corollary}{Corollary}{Corollaries}
\newtheorem{cor}[thm]{Corollary}
    \crefname{cor}{Corollary}{Corollaries}
\newtheorem{proposition}[thm]{Proposition}
    \crefname{proposition}{Proposition}{Propositions}
\newtheorem{prop}[thm]{Proposition}
    \crefname{prop}{Proposition}{Propositions}

    \crefname{problem}{Problem}{Problems}

    \crefname{conjecture}{Conjecture}{Conjectures}
\newtheorem{observation}[thm]{Observation}
    \crefname{observation}{Observation}{Observations}
\newtheorem{question}[thm]{Question}
    \crefname{question}{Question}{Questions}
\newtheorem*{claim*}{Claim}

    \crefname{claim}{Claim}{Claims}
\newtheorem{clm}{Claim}[]
    \crefname{clm}{Claim}{Claims}
\newtheorem*{case*}{Case}

    \crefname{case}{Case}{Case}
\newtheorem{thm-intro}{Theorem}[]
    \crefname{thm-intro}{Theorem}{Theorems}
\newtheorem{conj-intro}[thm-intro]{Conjecture}
    \crefname{conj-intro}{Conjecture}{Conjectures}
\newtheorem{question-intro}[thm-intro]{Question}
    \crefname{question-intro}{Question}{Questions}

\theoremstyle{definition}

    \crefname{definition}{Definition}{Definitions}

    \crefname{remark}{Remark}{Remarks}

    \crefname{remarks}{Remarks}{Remarks}

    \crefname{situation}{Situation}{Situations}

    \crefname{construction}{Construction}{Constructions}

    \crefname{construction}{Example}{Examples}
\newtheorem*{example*}{Example}

\newenvironment{subproof}[1][Proof.]{%
    \begin{proof}[{#1}]%
        }{%
    \end{proof}}

\renewcommand{\preceq}{\preccurlyeq}

\DeclareFontFamily{U}  {MnSymbolC}{}
\DeclareSymbolFont{MnSyC}         {U}  {MnSymbolC}{m}{n}
\DeclareFontShape{U}{MnSymbolC}{m}{n}{
    <-6>  MnSymbolC5
   <6-7>  MnSymbolC6
   <7-8>  MnSymbolC7
   <8-9>  MnSymbolC8
   <9-10> MnSymbolC9
  <10-12> MnSymbolC10
  <12->   MnSymbolC12}{}
\DeclareMathSymbol{\powerset}{\mathord}{MnSyC}{180}

\let\emptyset=\varnothing
\let\setminus=\smallsetminus

\newcommand*{\abs}[1]{\ensuremath{{\left\lvert {#1} \right\rvert}}}
\newcommand*{\gen}[1]{\ensuremath{{\left\langle {#1} \right\rangle}}}

\newcommand*{\weights}{\ensuremath{\mathbb{R}_{\geq 0}^{V(G)}}}

\newcommand{\norm}[1]{\ensuremath{\left\lVert {#1} \right\rVert}}

\hypersetup{
    pdftitle={Sharing tea on a graph},
    pdfauthor={J. Pascal Gollin, Kevin Hendrey, Hao Huang, Tony Huynh, Bojan Mohar, Sang-il Oum, Ningyuan Yang, Wei-Hsuan Yu, and Xuding Zhu}
}

\begin{document}
\author[Gollin]{J.~Pascal Gollin$^1$}
\address{$^1$FAMNIT,
University of Primorska, Koper, Slovenia}
\email{\tt pascal.gollin@famnit.upr.si}
\author[Hendrey]{Kevin Hendrey$^2$}
\address{$^2$School of Mathematics, Monash University, Melbourne, Australia}
\email{\tt kevinhendrey@gmail.com}

\author[Huang]{Hao Huang$^3$}
\address{$^3$Department of Mathematics, National University of Singapore,~Singapore}
\email{\tt huanghao@nus.edu.sg}

\author[Huynh]{Tony Huynh$^4$}
\address{$^4$Discrete Mathematics Group, Institute for Basic Science (IBS), Daejeon, South Korea}
\email{\tt tony.@ibs.re.kr}

\author[Mohar]{Bojan Mohar$^{5,6}$}
\address{$^5$Department of Mathematics, Simon Fraser University, Burnaby, Canada}
\address{$^6$Faculty of Mathematics and Physics, University of Ljubljana, Ljubljana, Slovenia}
\email{\tt mohar@sfu.ca}

\author[Oum]{Sang-il Oum$^{4,7}$}
\address{$^7$Department of Mathematical Sciences, KAIST,  Daejeon,~South~Korea}
\email{\tt sangil@ibs.re.kr}

\author[Yang]{Ningyuan Yang$^{8}$}
\address{$^8$School of Mathematical Sciences, Fudan University,  Shanghai, China}
\email{\tt nyyang23@m.fudan.edu.cn}

\author[Yu]{Wei-Hsuan Yu$^{9}$}
\address{$^9$Department of Mathematics, National Central University, Chungli, Taoyuan, Taiwan}
\email{\tt u690604@gmail.com}

\author[Zhu]{Xuding Zhu$^{10}$}
\address{$^{10}$School of Mathematical Sciences, Zhejiang Normal University, Jinhua, China}
\email{\tt xdzhu@zjnu.edu.cn}

\title[Sharing tea on a graph]{Sharing tea on a graph}

\begin{abstract}
    Motivated by the analysis of consensus formation in the Deffuant model for social interaction, we consider the following procedure on a graph~$G$.  
    Initially, there is one unit of  tea at a fixed vertex~${r \in V(G)}$, and all other vertices have no tea.  
    At any time in the procedure, we can choose a connected subset of vertices~$T$ and equalize the amount of tea among vertices in~$T$.  
    We prove that if~${x \in V(G)}$ is at distance~$d$ from~$r$, then~$x$ will have at most~$\frac{1}{d+1}$ units of tea during any step of the procedure. 
    This bound is best possible and answers a question of Gantert. 

    We also consider arbitrary initial weight distributions. 
    For every finite graph~$G$ and~${w \in \weights}$, we prove that the set of weight distributions reachable from~$w$ is a compact subset of~$\weights$.  
\end{abstract}



\date{}

\maketitle

\section{Introduction}	
\label{sec:intro}

We consider the following procedure on a (possibly infinite) graph~$G$.
Let~$r$ be a fixed vertex of~$G$ and assign one unit of tea to~$r$ and no tea to the other vertices of~$G$.  
A \emph{sharing move} on~$G$ chooses a connected finite set~$T$ of vertices and equalizes the amount of tea among vertices in~$T$. 
A natural question is to determine the maximum amount of tea that we can move to a vertex~$v$  via a sequence of sharing moves. 

To describe the problem rigorously, we define a \emph{weight distribution} on a graph~$G$ as a function~${w \in \weights}$, where the weight~${w(v)}$ of a vertex~$v$ describes the amount of tea at~$v$. 
We denote the set of weight distributions on~$G$ by~$\mathbb{R}_{\geq 0}^{V(G)}$. 
For each~${S \subseteq V(G)}$, we let ${w(S) \coloneqq \sum_{x \in S} w(x)}$, 
and let~${\mathds{1}_S \in \weights}$ denote the weight distribution on~$G$ defined by~${w(x) = 1}$ if~${x \in S}$ and~${w(x) = 0}$ if~${x \in V(G) \setminus S}$. 
For a finite nonempty subset~${T \subseteq V(G)}$ inducing a connected subgraph, a \emph{sharing move} on~$T$ changes~$w$ to~$w'$, where~${w'(v) \coloneqq \frac{w(T)}{\abs{T}}}$ for all~${v \in T}$, and~${w'(v) \coloneqq w(v)}$ for all~${v \notin T}$. 
We say that~${w' \in \weights}$ is \emph{reachable} from~${w \in \weights}$ if there is a finite sequence of sharing moves that changes~$w$ to~$w'$. 

If~${T = \{x,y\}}$ is the set containing the ends of a single edge, then a sharing move on~$T$ is called an \emph{edge-sharing move}. 
We say that~${w' \in \weights}$ is \emph{edge-reachable} from~${w \in \weights}$ if there is a finite sequence of edge-sharing moves that changes~$w$ to~$w'$. 
For~${u,v \in V(G)}$, we let~${d_G(u,v)}$ be the length of a shortest path from~$u$ to~$v$ in~$G$. 

The following is our first main theorem.

\begin{theorem}
    \label{thm:main}
    Let~$G$ be a connected graph and let~${r,v\in V(G)}$. 
    If~${w \in \weights}$ is edge-reachable from~$\mathds{1}_{\{r\}}$, then
    \[
        w(v) \leq \frac{1}{d_G(r,v)+1}.
    \]
\end{theorem}

This answers an open problem raised by Nina Gantert~\cite{gantert2024}. 
It is easy to show that if edge-sharing moves are chosen uniformly at random, then the weight distribution will converge to the uniform distribution. 
The precise mixing time of this Markov chain has received considerable attention. 
See~\cite{MSW22} and the references therein, which trace the mixing time question back to a question of Jean Bourgain from the early 1980s. For us, the precise mixing time is not needed.  The  fact that edge sharing moves can approximate sharing moves is already enough to imply the following corollary of~\Cref{thm:main}. 
\

\begin{cor}
    \label{cor:main}
    Let~$G$ be a connected graph and let~${r,v\in V(G)}$. 
    If~${w \in \weights}$ is reachable from~$\mathds{1}_{\{r\}}$, then
    \[
        w(v) \leq \frac{1}{d_G(r,v)+1}.
    \]
\end{cor}

\begin{proof}
By~\Cref{thm:main} and the fact that edge sharing moves can approximate sharing moves, we have
\[
        w(v) \leq \frac{1}{d_G(r,v)+1}+\epsilon,
    \]
for every $\epsilon>0$. Thus, we are done by letting $\epsilon \to 0$.   
\end{proof}

The upper bound in~\Cref{cor:main} is best possible because we can perform a single sharing move on the vertex set of a shortest path from~$r$ to~$v$.

The edge-sharing procedure was introduced by H\"{a}ggstr\"{o}m~\cite{Haggstrom2012}, who called it \emph{Sharing a drink} (SAD). 
In~\cite{Haggstrom2012}, the SAD-procedure is used to analyze the \emph{Deffuant model}~\cite{DNAW2000} for consensus formation in social networks. 
In this model, opinions are represented by real numbers, and when two agents interact, they move their opinions closer to each other, provided that their opinions are not too far apart. 
The model of edge-sharing also has applications in water resource management \cite{HH2015, HH2019}. 

Originally, the SAD-procedure was only considered on the two-way infinite path, but it clearly generalizes to any graph~(\cite{HH2015, HH2019}). 
H\"{a}ggstr\"{o}m~\cite{Haggstrom2012} proved~\Cref{thm:main} in the special case when~$G$ is the two-way infinite path, and Shang~\cite{shang} generalized it to all infinite $k$-regular trees. 

We believe that the SAD-procedure is of independent interest, and as far as we know,
~\Cref{thm:main} is the only result on the SAD-procedure for arbitrary graphs. 
Roughly speaking,~\Cref{thm:main} provides a precise limit that a malicious agent can spread an extreme opinion over any social network in the Deffuant model. 

We present two proofs of~\Cref{thm:main}. 
The first proof is shorter and uses a stronger inequality for the sum of weights for an arbitrary set, which allows us to study a single-source multiple-target problem. 
The second proof has a different viewpoint, based on proving a stronger statement for a multiple-source single-target problem. The second proof currently only works for finite graphs, but may be easier to generalize to arbitrary weight distributions. 

In~\Cref{sec:appendix,sec:appendix2}, we give strengthened versions of our two proofs which directly work for sharing moves instead of edge sharing moves.  In other words, we also give two direct proofs of~\Cref{cor:main}, which avoid using~\Cref{thm:main}.

We also consider the case of arbitrary initial weight distributions. 
It is not obvious that the maximum amount of tea at a vertex (maximum weight at a vertex) can always be achieved after a finite number of sharing moves. 
Indeed, this turns out to be false for infinite graphs, but true for finite graphs. 
The result for finite graphs follows from a much more general result. 
We prove that for every finite graph~$G$, the set of weight distributions reachable from a fixed weight distribution~${w \in \weights}$ is a compact subset of~$\weights$ under the standard topology on~$\mathbb{R}_{\ge0}^{V(G)}$.
These results are presented in~\Cref{sec:reachable}. 
We present some duality theorems in~\Cref{sec:duality} and some concluding remarks in~\Cref{sec:discussions}.

\section{The multiple target approach}
\label{sec:first}

For integers~$a$ and~$b$, we let~${[a,b] \coloneqq \{z \in \mathbb{Z} \mid a \leq z \leq b\}}$ and~${[a] \coloneqq [1,a]}$. 
Let~$G$ be a graph, let~${w \in \weights}$, and let~${r \in V(G)}$ be a fixed vertex. We will use~${w(G)}$ as shorthand for~${w(V(G))}$. 
For a finite set~${S \subseteq V(G)}$, let~$\rho(S)$ denote~${\prod_{x \in S} \frac{d_G(r,x)}{d_G(r,x)+1}}$. 
Note that $\rho(S)=0$ whenever $r\in S$, because $d_G(r,r)=0$.
We say that~${w \in \weights}$ is \emph{$r$-feasible}, if for every finite set~${S \subseteq V(G)}$, 
\begin{equation}
    \label{eq:set}
     w(S) \leq w(G) \left( 1 - \rho(S) \right).
\end{equation}
Clearly, the weight distribution~$\mathds{1}_{\{r\}}$ is $r$-feasible. 
Therefore, the following proposition immediately implies~\Cref{thm:main} by taking~${S \coloneqq \{v\}}$ in~(\ref{eq:set}). 

\begin{proposition}
    \label{prop:main}
    Let~$G$ be a connected graph, let~${r \in V(G)}$, and let~${w_0 \in \weights}$ be $r$-feasible.  
    Then every~${w \in \weights}$ on~$G$ that is edge-reachable from~$w_0$ is $r$-feasible. 
\end{proposition}

\begin{proof}
    By normalizing, we may assume that~${w_0(G) = 1}$. 
    Let~${(w_0, \dots, w_t)}$ be a sequence of weight distributions on~$G$ such that~${w_t = w}$ and~$w_i$ is obtained from~$w_{i-1}$ by a single edge-sharing move for each~${i \in [t]}$. 
    
    We proceed by induction on~$t$. 
    We may assume that~${t \geq 1}$; 
    otherwise there is nothing to prove. 
    Let~${w' = w_{t-1}}$. 
    By induction, $w'$ is $r$-feasible. 
    Suppose that~$w$ is obtained from~$w'$ by sharing on~${xy \in E(G)}$. 
    We will show that~$w$ is also $r$-feasible. 
    Let~$S$ be an arbitrary finite subset of~${V(G)}$. 
    If~${\{x,y\} \subseteq S}$ or~${\{x,y\} \cap S = \emptyset}$, then~${w(S) = w'(S)}$. 
    Thus, by symmetry, we may assume~${x \in S}$ and~${y \notin S}$.  
    Therefore,

    \begin{align*}
        w(S)
        &= \frac{1}{2} \left( w'(S-x) + w'(S+y) \right) \\
        &\leq \frac{1}{2} \left( 1- \rho(S-x) + 1-\rho(S+y) \right) \\
        &= 1 - \frac{1}{2} \rho(S-x) \left(1+\frac{d_G(r,x)}{d_G(r,x)+1} \cdot \frac{d_G(r,y)}{d_G(r,y)+1}\right) \\
        & \leq 1 - \rho (S-x)\frac{d_G(r,x)}{d_G(r,x)+1} \\
        &=1- \rho(S),
    \end{align*}
    where the second last line follows from the easy fact that~${1+\frac{a}{a+1}\frac{b}{b+1} \geq \frac{2a}{a+1}}$ for all integers~${a,b \geq 0}$ with~${\abs{a-b} \leq 1}$. 
    Note that~${\abs{d_G(r,x)-d_G(r,y)} \leq 1}$ since~${xy \in E(G)}$. 
    Thus, $w$ is also $r$-feasible.  
\end{proof}

\section{The multiple source approach}
\label{sec:second}

Throughout this section,~$G$ is a fixed finite connected graph with~$n$ vertices and~$v$ is a fixed vertex of~$G$. 

Starting with a weight distribution~$w$ on~$G$, our goal is to maximize the amount of tea at~$v$ via a sequence of edge-sharing moves. 
Let
\begin{equation} \label{eq:wstarreach}
    w^*(x) \coloneqq \sup \{ w'(x) \mid w' \textnormal{ is edge-reachable from } w \}. 
\end{equation}

As noted in the introduction, edge-sharing moves can approximate sharing moves.  Therefore,
\begin{equation} \label{eq:wstar}
    w^*(x) = \sup \{ w'(x) \mid w' \textnormal{ is reachable from } w \}. 
\end{equation}

For example, if~${w \coloneqq \mathds{1}_{\{r\}}}$, then~${w^*(x) = \frac{1}{d_G(r,x)+1}}$ by~\Cref{prop:main}. 
Note that if~$w'$ is edge-reachable from~$\mathds{1}_{\{r\}}$, then~${w'(x)}$ is a dyadic rational. 
Therefore, the supremum in~\Cref{eq:wstarreach} is not always attained. 
In contrast, we will show that the supremum in~\Cref{eq:wstar} is always attained for all finite graphs in~\Cref{sec:reachable}. 

For~${w \in \weights}$, for a sequence~$(x_1, \dots, x_t)$ of distinct vertices of ${V(G) \setminus \{v\}}$, and for~${a \in \mathbb{R}_{\geq 0}}$, we define
\[
        f_{(x_1, \ldots,x_t);w}(a) 
        \coloneqq \left(\prod_{i=1}^t\frac{d_G(x_i,v)}{d_G(x_i,v)+1} \right)a 
        + \sum_{i=1}^t \left(\prod_{j=i+1}^t \frac{d_G(x_j,v)}{d_G(x_j,v)+1} \right)\frac{w(x_i)}{d_G(x_i,v)+1}.
\]

It is easy to check that we have the following recursion, 
\begin{equation} \label{eq:recursion}
    f_{(x_1, \dots, x_t); w}(a) = f_{(x_1, \dots, x_{t-1}); w}(a) + \frac{w(x_t)-f_{(x_1, \dots, x_{t-1}); w}(a)}{d_G(x_t, v)+1},
\end{equation}
with initial condition~${f_{\emptyset;w}(a) = a}$. 
The idea of \eqref{eq:recursion} is that 
$f_{(x_1, \dots, x_t); w}(a)$ is the amount of tea at~$v$
starting with~$a$ units of tea at~$v$ and then performing a sequence of sharing moves on $V(P_1)$, $V(P_2)$, $\ldots$, $V(P_t)$, 
where~$P_i$ is a shortest path from~$v$ to~$x_i$ for each~${i \in [t]}$
under the assumption that all vertices of $P_i$ except $x_i$ have the same amount of tea and $w(x_i)$ is the amount of tea at $x_i$ 
right before performing the sharing move on $V(P_i)$.

Note that~${f_{(x); w}(a) > a}$ if~${w(x) > a}$, 
that~${f_{(x); w}(a) < a}$ if~${w(x) < a}$, 
and that~${f_{(x); w}(a) = a}$ if~${w(x) = a}$.

For~${w \in \weights}$, an enumeration~${\sigma = (x_1,x_2, \ldots, x_{n-1})}$ of~${V(G) \setminus \{v\}}$, and~${a \in \mathbb{R}_{\geq 0}}$, we define for all~${i \in [n-1]}$, 
\begin{align*}
    \phi_{i,w}(\sigma,  a) &\coloneqq f_{(x_i, \ldots, x_{n-1}); w}(a) \quad \textnormal{ and } \quad
    \phi_{n,w}(\sigma,  a) \coloneqq a. 
    \intertext{We further define }
    \phi_{\sigma; w}(a) &\coloneqq \max\{\phi_{i,w}(\sigma,  a) \mid i \in [n] \} \}, \ \text{ and }\\
    \phi_w(a) &\coloneqq \max\{\phi_{\sigma; w}(a) \mid \sigma \text{ is an enumeration of $V(G) \setminus \{v\}$} \}.
\end{align*}
We are going to show that $w^*(v)\le \phi_w(w(v))$. In other words, no sequence of sharing moves can send more tea to~$v$ than the sequence of sharing moves on the shortest paths from~$v$ to some vertices in~$V(G) \setminus \{v\}$ in the some order.

The following lemma can be proved inductively from the definitions, and the proof is omitted. 

\begin{lem}
    \label{lem-increasing} 
    Let~${w \in \weights}$. 
    \begin{enumerate}[label=(\roman*)]
        \item The function~${f_{(x); w}(a)}$ is a strictly increasing function of~$a$ for each~${x \in V(G)\setminus \{v\}}$.
        \item For each enumeration~${\sigma}$ of ${V(G) \setminus \{v\}}$ and each~${i \in [n]}$, the function~$\phi_{i,w}(\sigma, a)$ is a strictly increasing function of~$a$. 
        \qed
    \end{enumerate}
\end{lem}

\begin{lem}
    \label{lem-order}
    For each~${w \in \weights}$ and each pair~${(x_1,x_2)}$ of distinct vertices in ${V(G) \setminus \{v\}}$, we have 
    ${f_{(x_1, x_2); w}(a) \leq f_{(x_2, x_1); w}(a)}$ if and only if~${w(x_1) \geq w(x_2)}$. 
\end{lem}

\begin{proof}
    By~\Cref{eq:recursion},
    \begin{align*}
        f_{(x_1,x_2); w}(a) 
        &= a + \frac{w(x_1)-a}{d_G(x_1,v)+1} + \frac{w(x_2) - (a + \frac{w(x_1) - a}{d_G(x_1,v)+1})}{d_G(x_2,v)+1}\\
        &= a + \frac{w(x_1)-a}{d_G(x_1,v)+1} + \frac{w(x_2)-a}{d_G(x_2,v)+1}-\frac{w(x_1)-a}{(d_G(x_1,v)+1)(d_G(x_2,v)+1)}.
    \intertext{Similarly, }
        f_{(x_2,x_1); w}(a)
        &= a + \frac{w(x_1)-a}{d_G(x_1,v)+1} + \frac{w(x_2)-a}{d_G(x_2,v)+1}-\frac{w(x_2)-a}{(d_G(x_1,v)+1)(d_G(x_2,v)+1)}. 
    \end{align*}
    Therefore, $f_{(x_1,x_2); w}(a) \leq f_{(x_2, x_1); w}(a) $ if and only if $w(x_1) \geq w(x_2)$. 
\end{proof}

\begin{lem}
    \label{lem-order2}
    Let~${w \in \weights}$, let~${a \in \mathbb{R}_{\geq 0}}$, and let~${\sigma = (x_1, x_2, \ldots, x_{n-1})}$ be an enumeration of ${V(G) \setminus \{v\}}$ 
    such that~${w(x_i) \leq w(x_{i+1})}$ for all~${i \in [n-2]}$. 
    Let~${\ell\in [n]}$. 
    If~${w(x_{i}) \leq a}$ for all~${i \in [\ell-1]}$ and~${w(x_j) \geq a}$ for all~${j \in [\ell, n-1]}$, 
    then~${\phi_w(a) = \phi_{\ell, w}(\sigma, a)}$.
\end{lem}

\begin{proof}
    For each~${i \in [\ell-1]}$, we have that~${f_{(x_i);w}(a) \leq a}$, and hence
    \[
        \phi_{i, w}(\sigma, a) = \phi_{i+1,w} (\sigma,f_{(x_i);w}(a)) \leq \phi_{i+1, w}(\sigma, a)
    \]
    by \Cref{lem-increasing}. 
    Similarly, for~${i \in [\ell+1, n]}$, we have~${f_{(x_{i-1});w}(a) \geq a}$, and hence 
    \[ 
        {\phi_{i-1, w}(\sigma, a)= \phi_{i,w} (\sigma,f_{(x_{i-1});w}(a)) \geq \phi_{i, w}(\sigma, a)}.
    \] 
    Thus~${\phi_{\sigma;w}(a) = \phi_{\ell, w}(\sigma,  a)}$. 
    
    It follows from the definitions and \Cref{lem-increasing,lem-order} that for every enumeration~$\sigma'$ of~${V(G) \setminus \{v\}}$ and every~${i \in [n]}$, we have~${\phi_{i,w}(\sigma,a) \geq \phi_{i,w}(\sigma',a)}$. 
    Hence, ${\phi_w(a) = \phi_{\ell, w}(\sigma,  a)}$, as desired. 
\end{proof}

\begin{lem}\label{lem:key-proof2}
    Let~${w \in \weights}$ and let~$w'$ be the weight distribution on~$G$ obtained from~$w$ by sharing on an edge~$xy$. 
    Then~${\phi_{w'}(w'(v)) \leq \phi_{w}(w(v))}$. 
\end{lem}

\begin{proof}
    Let~${\sigma=(v_1, \ldots, v_{n-1})}$ be an enumeration of~${V(G) \setminus \{v\}}$ such that~${w'(v_i) \leq w'(v_{i+1})}$ for all~${i \in [n-2]}$. 
    By \Cref{lem-order2}, ${\phi_{w'}(w'(v)) = \phi_{\ell, w'}(\sigma, w'(v))}$, where
    \[
        \ell \coloneq 1 + \abs{\{ i \in [n-1] \mid w'(v_i) < w'(v) \}}.
    \]

    First, consider the case~${v \notin \{x,y\}}$. 
    Then~${w(v) = w'(v)}$. 
    Since~${w'(x) = w'(y) = \frac{w(x)+w(y)}{2}}$, we may assume that~${x = v_i, y = v_{i+1}}$ and~${w(v_i) \leq w(v_{i+1})}$ for some~${i \in [n-2]}$. 
    We will show that~${\phi_{\ell, w'}(\sigma, w(v)) \leq \phi_{\ell, w}(\sigma, w(v))}$. 
    If~${i < \ell}$, then since~${w(v_j) = w'(v_j)}$ for all~${j \geq \ell}$, it follows that ${\phi_{\ell, w'}(\sigma, w(v)) = \phi_{\ell, w}(\sigma, w(v))}$. 
    So we may assume that~${\ell \leq i}$. 
    Let 
    \[
        a \coloneqq f_{(v_{\ell}, v_{\ell+1}, \ldots, v_{i-1}); w}(w(v)) = f_{(v_{\ell}, v_{\ell+1}, \ldots, v_{i-1}); w'}(w'(v)). 
    \]
    It suffices to show that~${f_{(v_i, v_{i+1}); w'}(a) \leq f_{(v_i, v_{i+1}); w}(a)}$. 
    Let~${d \coloneqq d_G(v, v_{i})+1}$ and let \linebreak ${d' \coloneqq d_G(v, v_{i+1})+1}$. 
    Then, since~${d \geq d'-1}$ and~${w(v_{i+1}) \geq w(v_i)}$, we have 
     \begin{align*}
        2&dd'(f_{(v_i, v_{i+1}); w}(a)-  f_{(v_i, v_{i+1}); w'}(a))\\
        &= 2d'w(v_i) + 2dw(v_{i+1}) -2w(v_i) - ((d+d')w(v_i) + (d+d')w(v_{i+1}) - (w(v_i)+w(v_{i+1}))) \\
        &= d'w(v_i)+dw(v_{i+1})-2w(v_i) - dw(v_i)-d'w(v_{i+1}) + w(v_i)+w(v_{i+1}) \\
        &= (d-d'+1)(w(v_{i+1})-w(v_i)) 
        \geq 0.
    \end{align*}
    
    So the only case left to consider is when~${x=v}$. 
    Since $w'(v) = w'(y)$ and $\ell$ is equal to the minimum index such that~${w'(v_\ell) \geq w'(v)}$, we   may assume~${y = v_{\ell}}$. 
    It follows from the definition that ${w'(v) = f_{(v_{\ell}); w'}(w'(v)) = f_{(v_{\ell}); w}(w(v)) }$. 
    Hence, 
    \[
        \phi_{w'}(w'(v)) 
        = \phi_{\ell, w'}(\sigma, w'(v)) 
        = \phi_{\ell, w}(\sigma, w(v)) 
        \leq \phi_w(w(v)). \qedhere
    \]
\end{proof}

The above lemmas imply the following.

\begin{proposition}
    \label{prop:main2} 
    For every weight distribution~$w$ on a finite connected graph~$G$ and every~${v \in V(G)}$, 
    \[
        {w^*(v) \leq \phi_{w}(w(v)) = f_{(x_1,x_2, \ldots, x_k); w}(w(v))},
    \]
    where~$\{x_1, x_2, \ldots, x_k\}=\{x \in V(G) \mid w(x) > w(v)\}$ and ${w(x_i) \leq w(x_{i+1})}$ for all~${i \in [k-1]}$. 
    \qed
\end{proposition}

\begin{cor}
    \label{cor-singlesourcesink}
    Let~$G$ be a finite connected graph,~${r \in V(G)}$, and ${w \coloneq \mathds{1}_{\{r\}}}$. 
    Then ${w^*(v) = \frac{1}{d_G(r, v)+1}}$ for all~${v \in V(G)}$.  
\end{cor}

\begin{proof}
    Let~${d \coloneqq d_G(r, v)}$. 
    Let~$T$ be the vertex set of a shortest path from~$r$ to~$v$, and let~$w'$ be obtained from~$w$ by sharing on~$T$. 
    Then~${w'(v) = \frac{1}{d+1}}$. 
    By~\Cref{eq:wstar},~${w^*(v) \geq \frac{1}{d+1}}$. On the other hand, by \Cref{prop:main2},  
    ${w^*(v) \leq f_{(r); w}(w(v)) = \frac{1}{d+1}}$. 
\end{proof}

Even though~\Cref{cor-singlesourcesink} is only stated for finite graphs, it can be used to prove the same result for infinite graphs.  

\begin{proof}[Proof of~\Cref{thm:main} from~\Cref{cor-singlesourcesink}]
    Let $G$ be an infinite connected graph, $r \in V(G)$, and $w=\mathds{1}_{\{r\}}$.  For each subgraph $H$ of $G$, let 
    \[
        w^*(v,H) = \sup \{ w'(v) \mid w' \textnormal{ is edge-reachable from $w$ using edges from $H$}\}. 
    \]
    Let~$\mathcal{H}$ be the collection of finite connected subgraphs of~$G$ containing~$r$ and~$v$. 
    By definition, if~$w'$ is edge-reachable from~$w$, then there is a \emph{finite} sequence of edge-sharing moves which changes~$w$ into~$w'$. 
    Therefore,
    \[
        w^*(v) = \sup \{w^*(v,H) \mid H \in \mathcal{H} \} \leq \sup \left\{ \frac{1}{d_H(r,v)} \,\middle|\, H \in \mathcal{H} \right\}=\frac{1}{d_G(r,v)}. \qedhere
    \]
\end{proof}

The following result demonstrates another situation in which the bound in \Cref{prop:main2} is tight. 
We write $K_{1,t}$ for the graph on $t+1$ vertices, consisting of a vertex $v$ of degree $t$ and $t$ other vertices of degree $1$, all adjacent to~$v$.
\begin{cor}
    \label{cor-finitestars}
    Let $G=K_{1,t}$ for a positive integer $t$ and let $v$ be a vertex of degree~$t$ in~$G$.
    For every~${w \in \weights}$ on~$G$, 
    \[
        {w^*(v) = \phi_{w}(w(v)) = f_{(x_1,x_2, \ldots, x_k); w}(w(v))},
    \]
    where~$\{x_1, x_2, \ldots, x_k\}=\{x \in V(G) \mid w(x) > w(v)\}$ and ${w(x_i) \leq w(x_{i+1})}$ for all~${i \in [k-1]}$.  
\end{cor}

\begin{proof}
    Consider the sequence ${(\{v,x_{1}\}, \{v,x_{2}\}, \dots, \{v,x_{k}\})}$ of edge-sharing moves. 
\end{proof}

\section{The space of reachable weight distributions}
\label{sec:reachable}

Let~$G$ be a graph, let~${w \in \weights}$, and let~${v \in V(G)}$. In this section we study the set of weight distributions reachable from $w$, which we denote by $\mathcal{R}(w)$. The main result of this section is that $\mathcal{R}(w)$ is a compact 
subset of $\weights$ for all finite graphs $G$. At the end of this section we show that this is not necessarily true for infinite graphs.  

We may interpret each element of $\mathcal{R}(w)$ as a  linear operator defined by a doubly stochastic $V(G) \times V(G)$ matrix. To be precise,
given a sharing move~$T \subseteq V(G)$, we define the matrix~$A_T$ where~${A_T(x,y) = A_T(y,x) = \frac{1}{\abs{T}}}$ if~${x,y \in T}$, ${A_T(x,y) = A_T(y,x) = 0}$ if~${x \neq y}$ and~${\{x,y\} \not\subseteq T}$, and~${A_T(x,x) = 1}$ if~${x \notin T}$, as illustrated in \Cref{fig:matrix}.  
Then, applying the sharing move on~$T$ to a weight distribution~$w$ yields the weight distribution $A_Tw$.

\begin{figure}[htbp]
    \centering
    \[
    A_T = 
    \begin{pmatrix}
        \frac{1}{\abs{T}} & \dots & \frac{1}{\abs{T}} & 0 & \dots & 0\\
        \vdots & \ddots & \vdots & \vdots & \ddots & \vdots\\
        \frac{1}{\abs{T}} & \dots & \frac{1}{\abs{T}} & 0 & \dots & 0\\
        0 & \dots & 0 & 1 & \dots & 0\\
        \vdots & \ddots & \vdots & \vdots & \ddots & \vdots \\
        0 & \dots & 0 & 0 & \dots & 1
    \end{pmatrix}
    \]
    \caption{The matrix~$A_T$ that represents the sharing move on~$T$. }
    \label{fig:matrix}
\end{figure}

Hence, a sequence of sharing moves~$(T_1, \dots, T_\ell)$ corresponds to the product of the matrices~${A_{T_\ell}  \cdots  A_{T_1}}$. Since the product of doubly stochastic matrices is itself doubly stochastic, we obtain the following. 

\begin{lemma}
    \label{lem:matrices}
     For every finite graph~$G$, there is a countable set~$\mathcal{R}$ of doubly stochastic $V(G) \times V(G)$ matrices such that for every~$w \in \mathbb{R}_{\geq 0}^{V(G)}$, we have~$\mathcal{R}(w) = \{ Aw \mid A \in \mathcal{R} \}$. \qed
\end{lemma}

Let~$\norm{w}$ denote the Euclidean norm of~$w$.  

\begin{observation}
    \label{obs:normdecreasing}
    For all~${A \in \mathcal{R}}$ and~${w \in \mathbb{R}^{V(G)}}$, ${\norm{Aw} \leq \norm{w}}$,  with equality if and only if $Aw=w$.
\end{observation}

\begin{proof}
Let $w \in \mathbb{R}_{\geq 0}$ and $T$ be a finite connected subset of $V(G)$.  Let $a:=\frac{\sum_{t \in T} w(t)}{|T|}$.  By Jensen's inequality, $\sum_{t \in T} w(t)^2 \geq |T|a^2$, with equality if and only if $w(t)=a$ for all $t \in T$.  
\end{proof}

We now show that every infinite sequence of sharing moves has a well-defined limit.  We will use the following theorem of Amemiya and Ando~\cite{AA1965}, which holds in the more general setting of arbitrary Hilbert spaces and operators which satisfy the conclusion of~\Cref{obs:normdecreasing}. We state their theorem only in our restricted setting.

\begin{lemma}[Amemiya and Ando~\cite{AA1965}; see Dye~\cite{Dye1989}]\label{lem:aa}
    Let $G$ be a finite graph and let $T_1$, $T_2$, $\ldots$ be an infinite sequence of sharing moves of~$G$ such that every~$T_i$ appears infinitely many times in the sequence. For each $i \in \mathbb{N}$, let $S_i:=A_{T_i}A_{T_{i-1}}\cdots A_{T_{1}}$, viewed as a linear operator $\mathbb{R}^{V(G)} \to \mathbb{R}^{V(G)}$. Then $\lim_{i \to \infty}S_i$ is the projection~$\pi$ onto the subspace $\bigcap_{k=1}^\infty \{x\in \mathbb{R}^{V(G)}: A_{T_{k}} x= x\}$.
\end{lemma}
\begin{lem}\label{converge}
    Let~$G$ be a finite graph, and~${(w_i)_{i \in \mathbb{N}}}$ be a sequence of weight distributions on~$G$ such that~$w_i$ is reachable from~$w_{i-1}$ for all integers~${i>1}$. 
    Then~${w \coloneqq \lim_{i \to \infty} w_i}$ exists. 
    Moreover, $w$ is reachable from~$w_1$. 
\end{lem}
Note that we allow $w_i=w_{i-1}$, since we can take a singleton set as a sharing move.

\begin{proof}[Proof of \Cref{converge}]
    Without loss of generality, we may assume that for each~${i \in \mathbb{N}}$, there is a single sharing move on~$T_{i}$ which takes~$w_i$ to~$w_{i+1}$.  Let $\mathcal{I}$ be the family of subsets of $V(G)$ which occur infinitely often in the sequence $(T_{i})_{i \in \mathbb{N}}$.    Since there are only finitely many subsets of $V(G)$ and each $X \notin \mathcal{I}$ occurs only finitely often in $(T_{i})_{i \in \mathbb{N}}$, there exists $N \in \mathbb{N}$ such that $T_i \in \mathcal{I}$  for all $i \geq N$. 
    For each $i \in \mathbb{N}$, let $S_i:=A_{T_{N+i}}A_{T_{N+i-1}}\cdots A_{T_{N+1}}$.  By \Cref{lem:aa}, $\lim_{i \to \infty} S_i$ is the projection $\pi$ onto the subspace $\bigcap_{T \in \mathcal{I}} \{x\in\mathbb{R}^{V(G)}: A_{T}x=x\}$. 
    Thus, $w:=\lim_{i \to \infty} w_i=\lim_{i \to \infty} S_iw_N=\pi w_N$.
    
    For the second part, let $H$ be the hypergraph with vertex set $V(G)$ and edge set $\mathcal{I}$. Let $H_1, \dots, H_c$ be the connnected components of $H$.  Instead of applying the infinite sequence of sharing moves from $w_N$, 
        we apply the finite sequence of sharing moves $V(H_1), \dots, V(H_c)$. Since $w=\pi w_N$, the resulting weight distribution is identical to $w$ and shows that $w$ is reachable from $w_1$.  
\end{proof}

We are now ready to prove that $\mathcal{R}(w)$ is compact. The main idea is to apply Zorn's lemma on a suitably defined poset.  

\begin{theorem}
    \label{thm:compactness}
    For every finite graph~$G$ and every weight distribution~$w$ on~$G$, the set~$\mathcal{R}(w)$ is a compact subset of~$\mathbb{R}_{\geq 0}^{V(G)}$. 
\end{theorem}

\begin{proof} 
    Since $\mathcal{R}(w)$ is bounded and $V(G)$ is finite, it suffices to show that $\mathcal{R}(w)$ is closed. 
    Let~$(w_i)_{i \in \mathbb{N}}$ be a convergent sequence of weight distributions on~$G$ in~$\mathcal{R}(w)$ and let ${\hat{w} \coloneqq \lim_{i \to \infty} w_i}$. It suffices to show that $\hat{w} \in \mathcal{R}(w)$.  
    We define~${\sigma \colon \mathbb{R}_{\geq 0}^{V(G)}} \to \mathbb{R}$ by
    \[
        \sigma(x) \coloneqq \inf \big\{ \norm{\hat{w} - u} \,\colon\, u \in \mathcal{R}(x) \big\}. 
    \]
    Since~${w_i \in \mathcal{R}(w)}$ for all~${i \in \mathbb{N}}$, we have that~${0 \leq \sigma(w) \leq \inf \{ \norm{\hat{w} - w_i} \,\colon\, i \in \mathbb{N} \} = 0}$.  Hence, ${\sigma(w) = 0}$. 
    \begin{claim*}
        $\sigma$ is continuous.
    \end{claim*}
    
    \begin{subproof}
        Let~${f \colon \mathbb{R}^{V(G)} \to \mathbb{R}}$ be defined by~${f(x) = \norm{x-\hat{w}}}$. 
        By the triangle inequality, $\abs{f(x)-f(y)} \leq \norm{x-y}$ for all~${x,y \in \mathbb{R}^{V(G)}}$. 
        In other words, $f$ is a Lipschitz function with Lipschitz constant~${K = 1}$. 
        By~\Cref{obs:normdecreasing}, for each~${A \in \mathcal{R}}$, the map~${x \mapsto Ax}$ is a Lipschitz function with Lipschitz constant~${K = 1}$. 
        Therefore, $\sigma$ is the pointwise infimum of a family~$\mathcal{F}$ of Lipschitz functions, each with Lipschitz constant~${K = 1}$. 
        It follows that the family~$\mathcal{F}$ is equicontinuous, and hence~$\sigma$ is continuous. 
    \end{subproof}
    
    Consider the set 
    \[
        P \coloneqq \{ u \in \mathcal{R}(w)  \mid \sigma(u) = 0 \}
    \]
    together with the partial order
    \[
        u \preceq u' \textnormal{ if and only if } u' \textnormal{ is reachable from } u.
    \]
    Note that~${(P, \preceq)}$ is a countable partially ordered set. 
    Let~$C$ be a chain in~$P$. 
    Since every countable linear order is a suborder of~$\mathbb{Q}$, there is a subset~$I$ of~$\mathbb{Q}$ and an order-preserving bijection~${\phi \colon I \to C}$. 
    Let~${(q_i)_{i \in \mathbb{N}}}$ be a non-decreasing sequence of elements in~$I$ such that~${\lim_{i \to \infty} q_i = \sup I}$.    
    
    By \Cref{converge}, there exists~${u \in \weights}$ such that~${\lim_{i \to \infty} \phi(q_i) = u}$ and~$u$ is reachable from~$\phi(q_i)$ for all $i \in \mathbb{N}$.  
    Hence, ${u \in \mathcal{R}(w)}$. 
    Moreover,
    \[
       0= \sigma(w) 
        = \lim_{i \to \infty} \sigma(\phi(q_i)) 
        = \sigma(\lim_{i \to \infty} \phi(q_i)) 
        = \sigma(u), 
    \]
    where the third equality follows from the continuity of~$\sigma$. 
    Therefore, ${u \in P}$.  
    
    For each~${c \in C}$, since~${\lim_{i \to \infty} q_i = \sup I}$, there is some~${n \in \mathbb{N}}$ such that~${c \preceq \phi(q_n)}$. 
    Thus, ${c \preceq \phi(q_n) \preceq u}$. 
    So, $u$ is an upper bound for~$C$. 
    
    We have proved that each chain in~$P$ has an upper bound. 
    By Zorn's Lemma, we conclude that~$P$ has a maximal element~$w'$. 
    Suppose~${\norm{\hat{w} - w'} > 0}$. 
    Since~${\sigma(w') = 0}$, and there are only finitely many sharing moves, there must be a vertex set~$T$ on which we can apply the sharing move such that~${w' \neq A_Tw'}$ and~${\sigma(w') = \sigma(A_Tw')}$. 
    However, this contradicts that~$w'$ is a maximal element of~$P$. 
    Therefore, ${\norm{\hat{w} - w'} = 0}$ and hence~${\hat{w} = w' \in \mathcal{R}(w)}$, as required.
\end{proof}

The compactness of~$\mathcal{R}(w)$ immediately implies the following corollary. 

\begin{corollary}
    \label{cor:supremum}
    Let~$G$ be a finite graph, let~${w \in \weights}$, and let~${f \colon \weights \to \mathbb{R}}$ be any continuous function. 
    Then there exists~${w' \in \mathcal{R}(w)}$ such that ${\sup_{u \in \mathcal{R}(w)} f(u) = f(w')}$. 
\end{corollary}

\Cref{cor:supremum} has several important consequences.  Recall that 
\[
    w^*(v) = \sup \{w'(v) \mid \text{$w' \in \mathcal{R}(w)$} \}.
\]
We say that a finite sequence~${(T_1, \dots, T_m)}$ of sharing moves is \emph{($w, v)$-optimal} if the weight distribution~$w'$ on~$G$ obtained from~$w$ by sequentially applying~${T_1, \dots, T_m}$ satisfies~${w'(v) = w^*(v)}$. 

\begin{corollary}
    \label{cor:wvoptimal}
    Let~$G$ be a finite graph, let~${w \in \weights}$ and let~${v \in V(G)}$. 
    Then there exists a $(w,v)$-optimal sequence.
\end{corollary}

\begin{proof}
    Apply~\Cref{cor:supremum} with~${f \colon \weights \to \mathbb{R}}$ defined by~${f(u) \coloneqq u(v)}$.
\end{proof}

Compactness of~${\mathcal{R}(w)}$ also implies that  the maximum amount of tea on a fixed set of vertices~$X$ can be achieved after a finite number of sharing moves.  

\begin{corollary}
    Let~$G$ be a finite graph,~${w \in \weights}$,~${X \subseteq V(G)}$, and~${s \coloneqq \sup_{u \in \mathcal{R}(w)} u(X)}$.  
    Then there exists~${w' \in \mathcal{R}(w)}$ such that~${w'(X) = s}$. 
\end{corollary}

\begin{proof}
        Apply~\Cref{cor:supremum} with~${f \colon \weights \to \mathbb{R}}$ defined by~${f(u) \coloneqq u(X)}$.
\end{proof}

However, for infinite graphs, ${(w,v)}$-optimal sequences do not always exist, even if $w^*(v)$ is finite.  

\begin{prop}
    \label{prop:infinitestar}
    Let~${G \coloneqq K_{1,\aleph_0}}$ be the countable star with centre~$v$ and leaves~${\{ v_i \mid i \in \mathbb{N} \}}$, and let~${w \in \weights}$ with ${w(v) = 0}$ and~${w(v_i) = 2^{-i}}$ for all $i \in \mathbb{N}$. 
    Then there is no $(w, v)$-optimal sequence.
\end{prop}

\begin{proof}
    Clearly~$w^*(v)$ is finite since~${w^*(v) \leq w(G) = 1}$. 
    For each subgraph~$H$ of~$G$, we let 
    \[
        w^*(v,H) \coloneqq \sup \{ w(v) \mid \text{$w$ is reachable from $w$ by sharing on connected subsets of $H$} \}. 
    \]
    If there is a ${(w, v)}$-optimal sequence~${(T_1, \dots, T_m)}$, then~${w^*(v) = w^*(v,H)}$ for some finite subgraph~$H$ of~$G$ (since~$m$ is finite and each~$T_i$ is finite). Let $H'$ be any finite subgraph of~$G$ which contains~$H$ properly. 
    By~\cref{cor-finitestars}, ${w^*(v, H') > w^*(v,H)}$.
    Thus,
    \[
        w^*(v) = w^*(v, H) < w^*(v, H') \leq w^*(v),
    \]
    which is a contradiction.  
\end{proof}

\Cref{prop:infinitestar} shows that~$\mathcal{R}(w)$ need not be compact for infinite graphs.   

\section{Duality}
\label{sec:duality}
 Let $G$ be a graph and ${w \in \mathbb{R}_{\geq 0}^{V(G)}}$.
In this section, we prove that the tea sharing procedure on $G$ satisfies a simple duality relation. For instance, we show that if $w \in \{0,1\}^{V(G)}$, then maximizing the amount of tea at a fixed vertex $v$ is the same as starting from $\mathds{1}_{\{v\}}$ and maximizing the total amount of tea on $X$, where $X$ is the support of $w$. 

 We begin by noting that every ${c \in \mathbb{R}^{V(G)}}$ can be viewed as a linear functional on~$\mathcal{R}(w)$ via the inner product 
\[
    \gen{c,u} \coloneqq \sum_{v\in V(G)} c(v)u(v).
\]

\begin{theorem} \label{thm:duality}
    Let~$G$ be a graph, ${c \in \mathbb{R}^{V(G)}}$, ${w \in \weights}$, and~${w' \coloneqq A_{T_k} \cdots A_{T_1}w \in \mathcal{R}(w)}$. 
    Then,
    \[
        \gen{c,w'} = \gen{w,c'},
    \]
    where~${c' \coloneqq A_{T_1} \cdots A_{T_k} c}$.
\end{theorem}

\begin{proof}
    Since each matrix $A_{T_i}$ is symmetric, 
    \[
       \gen{c,w'} 
       = \gen{c, A_{T_k} \cdots A_{T_1}w} 
       = \gen{w, A_{T_1} \cdots A_{T_k}c}. 
       \qedhere
    \]
\end{proof}

As an immediate corollary of~\Cref{thm:duality}, the following is our claimed duality relation. 

\begin{corollary} \label{cor:switch}
    Let~$G$ be a graph, and~${c,w \in \weights}$. 
    Then, 
    \[
        \sup_{w' \in \mathcal{R}(w)} \gen{c,w'} 
        = \sup_{c' \in \mathcal{R}(c)} \gen {w,c'}.
    \]
\end{corollary}

Roughly speaking, \Cref{cor:switch} shows that for all~${c,w \in \weights}$, the roles of $c$ and $w$ can be interchanged.

\begin{corollary}
        Let $G$ be a graph, $v \in V(G)$, $X \subseteq V(G)$ and~${w  \coloneqq \mathds{1}_X}$. 
    Then
    \[
        w^*(v) 
        = \sup\{u(X) \mid u \in \mathcal{R}(\mathds{1}_{\{v\}})\}.
    \] 
    
    Moreover, if~$X$ is finite and $G$ is connected, then
    \[
        w^*(v) \leq \abs{X} \left( 1 - \prod_{x \in X} \frac{d_G(v,x)}{d_G(v,x)+1} \right).
    \]
\end{corollary}

\begin{proof}
    For the first part, take $c=\mathds{1}_{\{v\}}$ and ${w  \coloneqq \mathds{1}_X}$ in~\Cref{cor:switch}.  Since $\mathds{1}_{\{v\}}$ is $v$-feasible, the second part follows from~\Cref{prop:main}.  
\end{proof}

For finite graphs, compactness of~${\mathcal{R}(w)}$ also implies that linear functionals attain their supremum on~${\mathcal{R}(w)}$.

\begin{corollary}
    Let~$G$ be a finite graph,~${w \in \weights}$,~${c \in \mathbb{R}^{V(G)}}$, and~${s \coloneqq \sup_{u \in \mathcal{R}(w)} \gen{c,u}}$.  
    Then there exists~${w' \in \mathcal{R}(w)}$ such that~${\gen{c,w'} = s}$. 
\end{corollary}

\begin{proof}
        Apply~\Cref{cor:supremum} with~${f \colon \weights \to \mathbb{R}}$ defined by~${f(u) \coloneqq \gen{c,u}}$.
\end{proof}

\section{Conclusion}
\label{sec:discussions}

We analyzed the SAD-procedure on an arbitrary (possibly infinite) graph~$G$. 
We proved that if the support of the initial weight distribution~${w \in \weights}$ is a single vertex~$r$, then the maximum amount of tea at~${v \in V(G)}$ can be achieved after just one sharing move. 
For arbitrary initial weight distributions, the situation is much more complicated, but we proved that for finite graphs~$G$, the maximum amount of tea at~${v \in V(G)}$ can be achieved after a finite number of sharing moves. 
However, it is unclear whether the length of a shortest ${(w,v)}$-optimal sequence is bounded by a computable function of~$\abs{V(G)}$ (even if~${w \in \{0,1\}^{V(G)}}$). 
In fact, we do not know how to prove that $(w,v)$-optimal sequences exist without using the Axiom of Choice. 
As a concrete problem, we ask the following.

\begin{question}
    \label{quest:finite}
    For every finite graph~$G$, every~${w \in \weights}$, and every~${v \in V(G)}$, is there always a ${(w,v)}$-optimal sequence whose length is at most $2^{\abs{V(G)}}$?
\end{question}

One approach to answering this question in the affirmative would be to show that there is always a $(w,v)$-optimal sequence that does not repeat any sharing move. 
The following  
result demonstrates that it is not the case that there is always a $(w,v)$-optimal sequence in which no move is a subset of a later move, which helps to illustrate the difficulty of \Cref{quest:finite}.

\begin{proposition}
    \label{prop:counterexample}
    There exists a finite tree~$G$, a weight distribution ${w \in \weights}$, and a vertex~${v \in V(G)}$, such that for every ${(w, v)}$-optimal sequence of sharing moves~${(T_1, \dots, T_m)}$, there are indices~$i$ and~$j$ with~${1 \leq i < j \leq m}$ such that~${T_i \subseteq T_j}$. 
\end{proposition}

\begin{proof}
    We claim that we may take~$G$, ${w \in \weights}$, and~$v$ as in~\Cref{fig:counterexample}. 
    Let~${S_1 \coloneqq \{t,v\}}$, ${S_2 \coloneqq \{s,t,u\}}$, and~${S_3 \coloneqq \{r,s,t,v\}}$. 
    Observe that performing the sequence~${(S_1, S_2, S_3)}$ of sharing moves yields~$132$ units of tea at~$v$. Therefore, $w^*(v) \geq 132$. Towards a contradiction suppose that~${(T_1, \dots, T_m)}$ is a ${(w, v)}$-optimal sequence such that $T_i \not\subseteq T_j$ for all $1 \leq i < j \leq m$. Without loss of generality, we may also assume that~${\abs{T_i} \geq 2}$ for all~${i \in [m]}$ and~${T_{i+1} \not\subseteq T_{i}}$ for all~${i \in [m-1]}$.  For each $i \in [m]$ let $w_i$ be the weight distribution obtained from $w$ by applying the sharing moves $T_1, \dots, T_i$ (in this order).

    \begin{clm} 
        $T_1 \notin \{\{t,v\},\{t,u\}\}$.    
    \end{clm}

    \begin{subproof} 
      First suppose $T_1=\{t,v\}$.
      Since $T_1 \not\subseteq T_i$ for all $i \in [2, m]$, we conclude that $v \notin T_i$ for all $i \in [2,m]$.  Therefore, $w^*(v)=w_1(v)=108<132$, which is a contradiction.  Next suppose $T_1=\{t,u\}$.  Since $T_1 \not\subseteq T_i$ for all $i \in [2, m]$, we conclude that $w^*(v)$ can be achieved in $G[\{r,s,t,v\}]$ with initial weights given by $w_1$.  However, it is easy to check by hand that this weight distribution is $r$-feasible, and so by~\Cref{prop:mainstrong} the maximum amount of tea at $v$ in this weighted graph is always at most $\frac{w_1(\{r,s,t,v\})}{4}=120<132$, which is a contradiction.
    \end{subproof}

        \begin{clm}
        ${w_2(x) \geq 64}$ for all~${x \in V(G)}$.
        \end{clm}

        \begin{subproof}
        Suppose not.
        By the previous claim, $T_1 \notin \{\{t,v\},\{t,u\}\}$.  Thus, ${w_1(x) \geq 64}$ for all~${x \in V(G)}$, unless $T_1=\{t,u,v\}$. Hence, ${w_2(x) \geq 64}$ for all~${x \in V(G)}$, unless $T_1=\{t,u,v\}$.  It is now easy to check that we are done, unless $T_2=\{s,t\}$.  Since $T_2 \not\subseteq T_i$ for all $i \in [3,m]$, we conclude that $w^*(v)$ can be achieved in $G[\{t,u,v\}]$ with initial weights given by $w_2$. However, it is easy to see that the amount of tea at $v$ in this weighted graph is always at most $w_2(v)=96<132$, which is a contradiction.      
        \end{subproof}
        
        Let~${w_2'(x) = w_2(x)-64}$ for all~${x \in V(G)}$. 
        It is easy to check by hand that~${w_2'}$ is $r$-feasible in all possible cases of the first two sharing moves.
        Let~$w_m$ and~$w_m'$ be obtained from~$w_2$ and~$w_2'$ respectively by applying the sharing moves~${T_3, \dots, T_m}$. 
        Since~$w_2'$ is $r$-feasible, by~\Cref{prop:main} $w_m'$ is also $r$-feasible. 
        Therefore, 
        \[
            w_m'(v) 
            \leq w_m'(G)\left(1- \frac{d_G(r,v)}{d_G(r,v)+1}\right) 
            = 268 \cdot \frac{1}{4}
            = 67. 
        \]
        Note that~${w_m(v) = w_m'(v)+64 \leq 131}$. 
        However, this contradicts that~${(T_1, \dots, T_m)}$ is $(w, v)$-optimal, since $w^*(v) \geq 132$.
\end{proof}

\begin{figure}[htbp]
    \begin{center}
        \begin{tikzpicture}
            [scale=2.0]
            \tikzset{vertex/.style = {circle, draw=black, fill=white!100, inner sep=0pt, minimum width=10pt}}
            \tikzset{vertex1/.style = {circle, draw, fill=white!100, inner sep=0pt, minimum width=10pt}}
            \tikzset{edge0/.style = {line width=1pt, black}}
            
            \node [vertex, label=below:{$\scriptstyle{300}$}] (v0) at (1,0) {$\scriptstyle{r}$};
            \node [vertex, label=below:$\scriptstyle{0}$] (v1) at (2,0) {$\scriptstyle{s}$};
            \node [vertex, label=below:$\scriptstyle{144}$] (v2) at (3,0) {$\scriptstyle{t}$};
            \node [vertex1, label=below:$\scriptstyle{72}$] (v3) at (4,0) {$\scriptstyle{v}$};
            \node [vertex, label=right:$\scriptstyle{72}$] (v4) at (3,1) {$\scriptstyle{u}$};
            
            \draw[edge0] (v0) -- (v1);
            \draw [edge0] (v1) -- (v2);
            \draw [edge0] (v2) -- (v3);
            \draw[edge0] (v2) -- (v4);
        \end{tikzpicture}
    \end{center}
    \caption{The tree $G$, weight distribution $w$, and vertex $v$ from the proof of~\Cref{prop:counterexample}.}
    \label{fig:counterexample}
 \end{figure}

Finally, we remark that~\Cref{thm:main} still holds in the more general setting where agents are allowed to choose how much tea to share.  
To be precise, let~$G$ be a graph and let~${w \in \weights}$. 
A \emph{quasi-edge-sharing move} on an edge~${xy \in E(G)}$ with~${w(x) \leq w(y)}$, consists of choosing some real number~$s$ with~${0 \leq s \leq \frac{w(y)-w(x)}{2}}$ and replacing~$w$ by~$w'$, 
where~${w'(z)=w(z)}$ for all~${z \notin \{x,y\}}$, ${w'(x)=w(x)+s}$, and~${w'(y)=w(y)-s}$. 
Note that quasi-edge-sharing moves generalize edge-sharing moves since we may choose~${s = \frac{w(y)-w(x)}{2}}$. 
We say that~$w'$ is \emph{quasi-edge-reachable} from~$w$ if~$w$ can be transformed into~$w'$ via a finite sequence of quasi-edge-sharing moves. 
It turns out that quasi-edge-sharing moves are no more effective at delivering tea to a target vertex than edge-sharing moves.

\begin{prop}
    Let~$G$ be a finite graph and let~${w \in \weights}$. 
    Every weight distribution on~$G$ that is quasi-edge-reachable from~$w$ is a convex combination of weight distributions that are edge-reachable from~$w$. 
    In particular, for every vertex~${x \in V(G)}$, 
    \[
        \sup \{ w'(x) \mid w' \textnormal{ is quasi-edge-reachable from } w \} = w^*(x). 
    \]
\end{prop}
\begin{proof}
    Let~$w'$ be a weight distribution on~$G$ which is quasi-edge-reachable from~$w$, 
    and let~${u_1v_1, u_2v_2, \dots, u_tv_t}$ by a minimal sequence of edges such that~$w'$ can obtained from~$w$ by performing quasi-edge sharing moves on these edges in order. 
    We proceed by induction on~$t$. 
    If~${t = 0}$, then~${w = w'}$. 
    Otherwise, there is a weight distribution~$w_1$ which can be obtained from~$w$ by applying at most~${t-1}$ quasi-edge-sharing moves and which can be transformed into~$w'$ by performing a quasi-edge-sharing move on~${u_tv_t}$. 
    By induction, $w_1$ is a convex combination of some set~$\mathcal{W}$ of weight distributions which are edge-reachable from~$w$.
    Let~$w_2$ be the weight distribution which is obtained from~$w_1$ by performing an edge-sharing move on~${u_tv_t}$, 
    and note that~$w'$ is a convex combination of~$w_1$ and~$w_2$.
    By performing an edge-sharing move using the edge~$u_tv_t$ to each weight distribution in~$\mathcal{W}$, we can express~$w_2$ and hence~$w'$ as a convex combination of weight distributions which are edge-reachable from~$w$. 
\end{proof}

\medskip

\section*{Acknowledgements}

This work was done while visiting the Academia Sinica, Taipei for the Pacific Rim Graph Theory Group Workshop in 2024 organized by Bruce Reed. 
Tony Huynh thanks Stefan Weltge from whom he learned about the tea sharing problem at the 26th Combinatorial Optimization Workshop in Aussois, France in 2024. 
We are also grateful to Nina Gantert and Timo Vilkas for providing valuable references and background information. 

J.~Pascal Gollin was supported by the Institute for Basic Science (IBS-R029-Y3). 
Kevin Hendrey, Tony Huynh, and Sang-il Oum were supported by the Institute for Basic Science (IBS-R029-C1). 
Hao Huang was supported in part by a start-up grant
at NUS and an MOE Academic Research Fund (AcRF) Tier 1 grant of Singapore. 
Bojan Mohar was supported through the ERC Synergy grant KARST (European Union, ERC, KARST, project number 101071836).
Wei-Hsuan Yu was supported by MOST of Taiwan under Grant109-2628-M-008-002-MY4. 
Xuding Zhu was supported by National Science Foundation of China with
grant numbers NSFC 12371359 and U20A2068. 

\printbibliography

 \appendix 
\section{\texorpdfstring{
\Cref{prop:main} with general sharing moves}
{Proposition 2.1 with general sharing moves}}
\label{sec:appendix}

We now present a direct proof of~\Cref{cor:main}.  This proof does not use~\Cref{thm:main}, nor does it require taking a limit (as in the proof of~\Cref{cor:main} given in \Cref{sec:intro}).  This proof is adapted from our proof of~\Cref{prop:main}.

An \emph{integer interval} is a set of the form~${[a,b]}$, for some integers~$a$ and~$b$.
Let~$G$ be a graph and let~${r \in V(G)}$. 
Note that if~$T$ is a finite nonempty subset of~${V(G)}$ which induces a connected subgraph, then the set~${\{ d_G(r,v) \mid v \in T \}}$ is an integer interval. 
We start our first proof by an inequality on integer intervals. 
\begin{lem}
    \label{lem:ineq}
    Let~$s$, $t$ be non-negative integers such that $s+t>0$.
    Let~$x_1$, $x_2$, $\ldots$, $x_s$, $y_1$, $y_2$, $\ldots$, $y_t$ be non-negative integers 
    such that~${\{x_1,x_2,\ldots,x_s,y_1,y_2,\ldots,y_t\}}$ is an integer interval.
    Then 
    \[ 
        \frac{s}{s+t}\prod_{i=1}^s x_i \prod_{j=1}^t y_i 
        + \frac{t}{s+t}\prod_{i=1}^s (x_i+1)\prod_{j=1}^t (y_j+1) \geq \prod_{i=1}^s x_i \prod_{j=1}^t (y_j+1).
    \] 
\end{lem}

\begin{proof}
    We may assume that $s,t>0$.
    Since permuting values of~${(x_1,\ldots,x_s,y_1,\ldots,y_t)}$ does not change the left-hand side, 
    we may assume that~${y_t \leq \cdots \leq y_1 \leq x_1 \leq \cdots \leq x_s}$ to maximize the right-hand side. 
    We may also assume that~${x_i > 0}$ for all~$i$. 
    By dividing both sides by~${\prod_{i=1}^s x_i \prod_{j=1}^t (y_j+1)}$, it is enough to prove that 
    \[ 
        \frac{s}{s+t}\prod_{j=1}^t \frac{y_i}{y_j+1}
        + \frac{t}{s+t}\prod_{i=1}^s \frac{x_i+1}{x_i} \geq 1. 
    \] 
    Since~${\frac{1+x}{x}}$ is a decreasing function for~${x > 0}$, we may assume that~${x_i = x_1 + i - 1}$ by increasing values of~$x_i$ whenever possible, as long as~${\{ x_1, x_2, \ldots, x_s, y_1, y_2, \ldots, y_t \}}$ is an integer interval.
    Therefore, 
    \[
        \prod_{i=1}^s \frac{x_i+1}{x_i} = \frac{x_1+s}{x_1}. 
    \]

    If~${x_1 \leq t}$, then~${\frac{t}{s+t} \frac{x_1+s}{x_1} \geq 1}$. 
    Therefore we may assume that~${x_1 > t}$. 

    Since~${\frac{y}{1+y}}$ is an increasing function for~${y > 0}$, we may assume that~${y_i = x_1 - i}$ by decreasing values of~$y_i$ whenever possible, as long as~${\{ x_1, x_2, \ldots, x_s, y_1, y_2, \ldots, y_t \}}$ is an integer interval. 
    Then 
    \[ 
        \prod_{i=1}^t \frac{y_i}{y_i+1} = \frac{x_1-t}{x_1}. 
    \]
    Now we deduce that 
    \(
        \frac{s}{s+t} \frac{x_1-t}{x_1}+ \frac{t}{s+t} \frac{x_1+s}{x_1} = 1
    \). 
    This proves the inequality. 
\end{proof}

We are now ready to prove the following strengthened version of~\Cref{prop:main}.

\begin{proposition}
    \label{prop:mainstrong}
    Let $G$ be a connected graph, let~${r \in V(G)}$, and let~${w_0 \in \weights}$ be $r$-feasible.  
    Then every~${w \in \weights}$ on~$G$ that is reachable from~$w_0$ is $r$-feasible. 
\end{proposition}

\begin{proof}
    By normalizing, we may assume that~${w_0(G) = 1}$. 
    Let~${(w_0, \dots, w_t)}$ be a sequence of weight distributions on~$G$ such that~${w_t = w}$ and~$w_i$ is obtained from~$w_{i-1}$ by a single sharing move for each~${i \in [t]}$. 
    We proceed by induction on~$t$. 
    We may assume that~${t \geq 1}$; 
    otherwise there is nothing to prove. 
    Let~${w' = w_{t-1}}$. 
    By induction, $w'$ is $r$-feasible. 
    Suppose that~$w$ is obtained from~$w'$ by sharing on a finite nonempty set~${T \subseteq V(G)}$. 
    
    We may assume that~${r \notin S}$. 
    Let~${s = \abs{S \cap T}}$ and~${t = \abs{T-S}}$. 
    By the definition of a sharing move, 
    \(
        {w(S) 
            = \frac{s}{s+t} w'(T) + w'(S-T) 
            = \frac{s}{s+t} w'(S\cup T) + \frac{t}{s+t} w'(S-T)}
    \). 
    For a finite set~${X \subseteq V(G)}$, recall~${\rho(X)=\prod_{x \in X} \frac{d_G(r,x)}{d_G(r,x)+1}}$. 
    Since~$w'$ satisfies~(\ref{eq:set}), 
    \[
        0 \leq w'(S\cup T) \leq 1 - \rho(S \cup T)
        \qquad\text{and}\qquad 
        0 \leq w'(S-T) \leq 1 - \rho(S-T). 
    \]
    We deduce that 
    \begin{align*}
        w(S)
        &\leq \frac{s}{s+t} \left( 1 - \rho(S \cup T) \right) + \frac{t}{s+t} \left( 1 - \rho(S-T) \right) \\
        &= 1 - \left( \frac{s}{s+t} \rho(S \cup T) + \frac{t}{s+t} \rho(S-T) \right) \\ 
        &= 1 - \rho(S) \left( \frac{s}{s+t} \rho(T-S) + \frac{t}{s+t} \frac{1}{\rho(T \cap S)} \right)\\
        &\leq 1 - \rho(S), 
    \end{align*}
    where the last inequality follows from \Cref{lem:ineq}. 
    Thus, $w$ is also $r$-feasible.  
\end{proof}

\section{\texorpdfstring
{\Cref{lem:key-proof2} with general sharing moves}
{Lemma 3.4 with general sharing moves}}
\label{sec:appendix2}

We now present a stronger variant of~\Cref{lem:key-proof2} with general sharing moves.
This will allow us to prove \cref{prop:main2} without needing to take a limit. 
Throughout this appendix,~$G$ is a fixed finite connected graph with $n$ vertices and $v$ is a fixed vertex of $G$. We use the same notation as in~\Cref{sec:second}.

\begin{lem}
    \label{lem-ineq}
    Let $t\ge2$ be an integer
    and 
    let $d_1,d_2,\ldots,d_t$ be positive integers such that $\{ d_1,d_2,\ldots,d_t\}$ is an integer interval. 
    Let $w_1,w_2,\ldots,w_t$ be non-negative reals such that 
    $w_i\le w_{i+1}$ for all $i\in [t-1]$.
    Then
    \[ 
    \sum_{i=1}^t \left(\prod_{j=1}^{i-1} (1+d_j)\right)\left( \prod_{k=i+1}^t d_k\right)w_i
    \ge \left( \sum_{i=1}^t \left(\prod_{j=1}^{i-1} (1+d_j)\right)\left( \prod_{k=i+1}^t d_k\right)\right)\frac1t \sum_{i=1}^t w_i .
    \] 
\end{lem}
\begin{proof}
    First, it is easy to see that 
    $\sum_{i=1}^t \prod_{j=1}^{i-1} (1+d_j) \prod_{k=i+1}^t d_k= \prod_{j=1}^t (1+d_j)-\prod_{k=1}^t d_k$
    and therefore permuting the values of $(d_1,d_2,\ldots,d_t)$ does not change the right-hand side of the inequality.

    We claim that we may assume that $d_\ell\le d_{\ell+1}$ for all $\ell\in [t-1]$. Suppose $d_\ell>d_{\ell+1}$ for some $\ell\in [t-1]$. 
    Observe that 
    \[ 
    (d_{\ell+1}w_\ell + (1+d_\ell)w_{\ell+1} )
    -(d_{\ell}w_\ell + (1+d_{\ell+1})w_{\ell+1})
    = (d_{\ell+1}-d_\ell)(w_\ell -w_{\ell+1})\ge 0
    \] 
    and therefore swapping the values of $d_\ell$ and $d_{\ell+1}$ does not increase the left-hand side.
    This proves the claim.

    Now observe that 
    $\left(\prod_{j=1}^{i-1} (1+d_j) \prod_{k=i+1}^t d_k\right) \le \left(\prod_{j=1}^{i} (1+d_j) \prod_{k=i+2}^t d_k\right)$ for each $i\in[t-1]$, 
    because $d_{i+1}\le 1+d_i$.
    By Chebyshev's sum inequality, we deduce the desired inequality.
\end{proof}
We are now ready to prove the following strengthened version of~\Cref{lem:key-proof2}.

\begin{lem} \label{lem:strongphi}
    Let~${w \in \weights}$ and let~$w'$ be the weight distribution on~$G$ obtained from~$w$ by sharing on a finite connected set $T$ of vertices.
    Then~${\phi_{w'}(w'(v)) \leq \phi_{w}(w(v))}$. 
\end{lem}

\begin{proof}
    Let~${\sigma=(v_1, \ldots, v_{n-1})}$ be an enumeration of $V(G)\setminus \{v\}$ such that~${w'(v_i) \leq w'(v_{i+1})}$ for all~${i \in [n-2]}$. 
    By \Cref{lem-order2}, ${\phi_{w'}(w'(v)) = \phi_{\ell, w'}(\sigma, w'(v))}$, where
    \[ \ell\coloneq 1+\abs{\{ i\in[n-1]\mid w'(v_i)<w'(v)\}}.\]
    We may assume that $\abs{T}>1$.

    First, consider the case~${v \notin T}$. 
    Then~${w(v) = w'(v)}$. 
    Since~${w'(x) = \frac{w(T)}{\abs{T}}}$ for every vertex~$x\in T$, we may assume that~${T=\{v_s,v_{s+1},\ldots,v_{s+t-1}\}}$ for some~${s \in [n-t]}$ with $t=\abs{T}$.
    We may also assume that~${w(v_{s+j-1}) \leq w(v_{s+j})}$ for all $j\in [t-1]$.
    We will show that~${\phi_{\ell, w'}(\sigma, w(v)) \leq \phi_{\ell, w}(\sigma, w(v))}$. 
    If~${s < \ell}$, then since~${w(v_j) = w'(v_j)}$ for all~${j \geq \ell}$, it follows that ${\phi_{\ell, w'}(\sigma, w(v)) = \phi_{\ell, w}(\sigma, w(v))}$. 
    So we may assume that~${\ell \leq s}$. 
    Let 
    \[
        a \coloneqq f_{(v_{\ell}, v_{\ell+1}, \ldots, v_{s-1}); w}(w(v)) = f_{(v_{\ell}, v_{\ell+1}, \ldots, v_{s-1}); w'}(w'(v)). 
    \]
    It suffices to show that~${f_{(v_s, v_{s+1},\ldots,v_{s+t-1}); w'}(a) \leq f_{(v_s, v_{s+1},\ldots,v_{s+t-1}); w}(a)}$. 
    For each $j\in [t]$, let $d_j\coloneqq d_G(v,v_{s+j-1})$
    and $w_j\coloneqq w(v_{s+j-1})$.
    Then, by \cref{lem-ineq}, we have 
    \begin{align*}
        f_{(v_s, v_{s+1},\ldots,v_{s+t-1});w}(a) 
        &= \left(\prod\limits_{i=1}^t\frac{d_i}{d_i+1} \right)a 
        + \sum_{i=1}^t \left(\prod\limits_{j=i+1}^t \frac{d_j}{d_j+1} \right)\frac{1}{d_i+1} w_i\\
        &=\left(\prod\limits_{i=1}^t\frac{d_i}{d_i+1} \right)a 
        + \frac{
        \sum\limits_{i=1}^t \left(\prod\limits_{j=1}^{i-1}(1+d_j)\right)\left(\prod\limits_{k=i+1}^t d_k \right) w_i }
        {\prod\limits_{i=1}^t (d_i+1)}\\
        &\ge \left(\prod\limits_{i=1}^t\frac{d_i}{d_i+1} \right)a 
        +  \frac{
        \sum\limits_{i=1}^t \left(\prod\limits_{j=1}^{i-1}(1+d_j)\right)\left(\prod\limits_{k=i+1}^t d_k \right) }
        {\prod\limits_{i=1}^t (d_i+1)}\frac{\sum\limits_{i=1}^t w_i}{t} \\
        &= f_{(v_s, v_{s+1},\ldots,v_{s+t-1});w'}(a).
    \end{align*}

    So the only case left to consider is when~${v\in T}$. 
    Since~${w'(x) = \frac{w(T)}{\abs{T}}}$ for every vertex~$x\in T$, we may assume that~${T=\{v,v_s,v_{s+1},\ldots,v_{s+t-2}\}}$ for some~${s \in [n-t+1]}$ with $t=\abs{T}$.
    We may also assume that~${w(v_{s+j-1}) \leq w(v_{s+j})}$ for all $j\in [t-2]$.
    By \Cref{lem-order2}, 
    \begin{align*}
        \phi_{w'}(w'(v)) 
        &= \phi_{s+t-1, w'}(\sigma, w'(v)) \\
        &= f_{(v_{s+t-1},v_{s+t-2},\ldots,v_{n-1});w'}(w'(v)) \\
        &= f_{(v_{s+t-1},v_{s+t-2},\ldots,v_{n-1});w}(w'(v))
        =\phi_{s+t-1,w}(\sigma,w'(v)).
    \end{align*}
    Since $\phi_{s+t-1,w}$ is strictly increasing, we may assume that $w(v)<w'(v)$, because otherwise $\phi_{w'}(w'(v))= \phi_{s+t-1,w}(\sigma,w'(v))\le \phi_{s+t-1,w}(\sigma,w(v))\le \phi_w(w(v))$.
    For each $j\in [t-1]$, let $d_j\coloneqq d_G(v,v_{s+j-1})$ and $w_j\coloneqq w(v_{s+j-1})$.
    Again, by \cref{lem-ineq}, we have 
    \begin{align*}
        f_{(v_s, v_{s+1},\ldots,v_{s+t-2});w}(w(v)) 
        &= \left(\prod_{i=1}^{t-1}\frac{d_i}{d_i+1} \right)w(v) 
        + \sum_{i=1}^{t-1} \left(\prod_{j=i+1}^{t-1} \frac{d_j}{d_j+1} \right)\frac{1}{d_i+1} w_i\\
        &=\left(\prod \limits_{i=1}^{t-1}\frac{d_i}{d_i+1} \right)w(v)
        + \frac{
        \sum \limits_{i=1}^{t-1} \left(\prod \limits_{j=1}^{i-1}(1+d_j)\right)\left(\prod \limits_{k=i+1}^{t-1} d_k \right) w_i }
        {\prod\limits_{i=1}^{t-1} (d_i+1)}\\
        &\geq \left(\prod_{i=1}^{t-1}\frac{d_i}{d_i+1} \right)w(v) 
        +  \frac{
        \sum \limits_{i=1}^{t-1} \left(\prod \limits_{j=1}^{i-1}(1+d_j)\right)\left(\prod \limits_{k=i+1}^{t-1} d_k \right) }
        {\prod \limits_{i=1}^{t-1} (d_i+1)}\frac{\sum \limits_{i=1}^{t-1} w_i}{t-1} \\
        &=  \left(\prod\limits_{i=1}^{t-1}\frac{d_i}{d_i+1} \right)w(v) 
        +  \left(1-\prod\limits_{i=1}^{t-1}\frac{d_i}{d_i+1} \right)
        \frac{\sum\limits_{i=1}^{t-1} w_i}{t-1}.
    \end{align*}
    Since $\{d_1,d_2,\ldots,d_{t-1}\}$ is an integer interval containing $1$,  we have 
    $\prod_{i=1}^{t-1}\frac{d_i}{d_i+1} \le \frac{1}{t}$.
    As $w(v)<w'(v)=
    \frac1t (w(v)+\sum_{i=1}^{t-1} w_i)$, we have $w(v)<\frac{\sum_{i=1}^{t-1}w_i}{t-1}$.
    Therefore, 
    \[ 
        \left(\prod_{i=1}^{t-1}\frac{d_i}{d_i+1} \right)w(v) 
        +  \left(1-\prod_{i=1}^{t-1}\frac{d_i}{d_i+1} \right)
        \frac{\sum \limits_{i=1}^{t-1} w_i}{t-1}
        \geq \frac{1}{t} w(v)+ \left(1-\frac1t\right) \frac{\sum \limits_{i=1}^{t-1} w_i}{t-1} = w'(v).
    \] 
    This implies that $f_{(v_s, v_{s+1},\ldots,v_{s+t-2});w}(w(v)) \geq w'(v) = f_{(v_s, v_{s+1},\ldots,v_{s+t-2});w'}(w'(v))$ and therefore 
    $\phi_w(w(v))\ge \phi_{s,w}(\sigma,w(v))\ge \phi_{s,w'}(\sigma,w'(v))=\phi_{w'}(w'(v))$, 
    where the last equality follows from \cref{lem-order2}.
\end{proof}
\color{black}

Combining~\Cref{lem:strongphi} with~\Cref{lem-increasing},~\Cref{lem-order}, and~\Cref{lem-order2} gives the desired proof of~\Cref{prop:main2}.

\end{document}